\documentclass{article}
\usepackage{amsmath, amsfonts, amsthm, amssymb, color}
\usepackage{graphicx}
\usepackage{float}
\usepackage{verbatim}

\numberwithin{equation}{section}

\newtheorem{thm}{Theorem}[section]
\newtheorem{lma}[thm]{Lemma}
\newtheorem{cor}[thm]{Corollary}

\newtheorem{prop}[thm]{Proposition}
\newtheorem{conj}[thm]{Conjecture}

\newtheorem{defi}[thm]{Definition}
\newtheorem{theorem}{Theorem}[section]
\newtheorem{lemma}[thm]{Lemma}
\theoremstyle{definition}
\newtheorem{remark}[thm]{Remark}

\renewcommand{\ge}{\geqslant}
\renewcommand{\le}{\leqslant}
\renewcommand{\geq}{\geqslant}
\renewcommand{\leq}{\leqslant}
\renewcommand{\H}{\text{H}}

\def\R{\mathbb{R}}

\def\iI{{\mathcal{I}}}
\def\Nn{N_{2^{-n}}}

\def\Nnk{N_{2^{-(n-k)}}}
\newcommand{\norm}[1]{{\left\lVert #1 \right\rVert}}

\newcommand{\eps}{\varepsilon}
\newcommand{\I}{I}
\newcommand{\J}{J}
\newcommand{\beq}{\begin{equation}}
\newcommand{\eeq}{\end{equation}}
\newcommand{\dimfc}{\overline{\dim}_B\,F_C}
\newcommand{\erase}[1]{}

\title{Inhomogeneous self-similar sets with overlaps}

\author{ Simon Baker$^1$, Jonathan M. Fraser$^2$ and Andr\'as M\'ath\'e$^3$ \\ \\
\emph{$^1$Department of Mathematics and Statistics,} \\ \emph{University of Reading,} \\ \emph{Reading,  RG6 6AX, UK} \\ Email: simonbaker412@gmail.com\\ \\
\emph{$^2$School of Mathematics,} \\ \emph{The University of Manchester,} \\ \emph{Manchester, M13 9PL, UK}\\ Email: jonathan.fraser@manchester.ac.uk\\ \\
\emph{$^3$Mathematics Institute,} \\ \emph{University of Warwick,} \\ \emph{Coventry, CV4 7AL, UK}\\ Email: A.Mathe@warwick.ac.uk\\ \\
 }

\begin{document}
\maketitle

\begin{abstract}
It is known that if the underlying iterated function system satisfies the open set condition, then the upper box dimension of an inhomogeneous self-similar set is the maximum of the upper box dimensions of the homogeneous counterpart and the condensation set.  First, we prove that this `expected formula' does not hold in general if there are overlaps in the construction.  We demonstrate this via two different types of counterexample: the first is a family of overlapping inhomogeneous self-similar sets based upon Bernoulli convolutions; and the second applies in higher dimensions and makes use of a spectral gap property that holds for certain subgroups of $SO(d)$ for $d\geq 3$.

We also obtain new upper bounds for the upper box dimension of an inhomogeneous self-similar set which hold in general. Moreover, our counterexamples demonstrate that these bounds are optimal. In the final section we show that if the \emph{weak separation property} is satisfied, ie. the overlaps are controllable, then the `expected formula' does hold. \\


\end{abstract}

\section{Inhomogeneous iterated function systems}

Inhomogeneous iterated function systems, introduced by Barnsley and Demko \cite{barndemko}, are natural generalisations of classical iterated function systems and consist of a classical (homogeneous) iterated function system (IFS) and a compact \emph{condensation set}.  Let $\{ S_i \}_{i \in \mathcal{I}}$ be a classical IFS, i.e. a finite collection of contracting self-maps on a compact subset of Euclidean space $X$ and let $C\subseteq X$ be the compact condensation set.  An elegant application of Banach's contraction mapping theorem yields that there is a unique non-empty compact set $F_C$ satisfying
\[
F_C =  \bigcup_{i \in \mathcal{I}} S_i (F_C) \ \cup \ C,
\]
which we refer to as the \emph{inhomogeneous attractor} of the inhomogeneous IFS.  Note that the attractors of classical (or homogeneous) IFSs are inhomogeneous attractors with condensation equal to the empty set.  It turns out that $F_C$ is equal to the union of $F_\emptyset$ and all images of $C$ by compositions of maps from the defining IFS.  This means that for countably stable dimensions, like the Hausdorff dimension $\dim_\H$, one immediately obtains
\[
\dim_\H F_C = \max \{ \dim_\H  F_\emptyset, \ \dim_\H C\},
\]
but establishing similar formulae for dimensions that are not countably stable is more challenging. As such it is natural to study the upper and lower box dimensions, $\overline{\dim}_\text{B}$ and $\underline{\dim}_\text{B}$. In what follows we will focus on inhomogeneous self-similar sets, i.e. inhomogeneous attractors where the defining contractions are similarities.  For this class of attractors it was shown in \cite{olseninhom} that if the defining system satisfies an `inhomogeneous strong separation condition', then the `expected formula' also holds for upper box dimension, i.e.
\begin{equation} \label{expected}
\overline{\dim}_\text{B} F_C = \max \{ \overline{\dim}_\text{B}  F_\emptyset, \ \overline{\dim}_\text{B} C\}.
\end{equation}
It was shown in \cite{fraser1} that the analogous formula fails for lower box dimension, even if one has good separation properties, and that $(\ref{expected})$ remains valid even if the separation condition from \cite{olseninhom} is relaxed to the open set condition (OSC), which, in particular, does not depend on $C$.  See \cite[Chapter 9]{falconer} for the definition of the OSC.  In \cite{fraser2} the problem was addressed for inhomogeneous self-\emph{affine} sets and in this context (\ref{expected}) does not generally hold for upper box dimension even if the OSC is satisfied.

In this paper we focus on the overlapping situation (i.e. without assuming the OSC) and prove that (\ref{expected}) does not hold in general by considering a construction based on number theoretic properties of certain Bernoulli convolutions (Section~\ref{resultssection}).  In Section~\ref{example},
relying on a specific spectral gap property of $SO(d)$ for $d\geq 3$,
we provide another construction in $\R^d$  where $\underline{\dim}_\text{B}\, F_C=d-1-\eps$ with $F_\emptyset$ and $C$ being singletons.

In Section~\ref{upperbounds} we show that these constructions (in a certain sense) give the highest possible value of $\overline{\dim}_\text{B} F_C$. In particular, we provide new general upper bounds for $\overline{\dim}_\text{B} F_C$ in terms of $\overline{\dim}_\text{B} F_\emptyset$ and $\overline{\dim}_\text{B}\,C$.

Finally, in Section~\ref{WSPsection},
we show that (\ref{expected}) does hold if the \emph{weak separation property} is satisfied, and make connections with the recent pioneering work of Hochman concerning homogeneous self-similar sets with overlaps \cite{hochman, hochman2}.

\subsection{Box dimensions and structure of inhomogeneous self-similar sets}

In this section we introduce some notation and state the main result of \cite{fraser1} to put the rest of the paper in context.  For $\delta>0$ and bounded $E \subseteq \mathbb{R}^n$, let $N_\delta(E)$ denote the number of half open cubes which intersect $E$ from the $\delta$-mesh imposed on $\mathbb{R}^n$ oriented with the co-ordinate axes.  Here a half open cube is taken to mean a product of half open intervals $[a,b)$.  The upper and lower box dimensions of $E$ can be defined in terms of $N_\delta(E)$ by
\[
\overline{\dim}_\text{B}E  \ = \ \limsup_{\delta \to 0} \frac{\log N_\delta(E)}{-\log \delta}
\]
and
\[
\underline{\dim}_\text{B}E \ = \ \liminf_{\delta \to 0} \frac{\log N_\delta(E)}{-\log \delta}
\]
respectively.  If the upper and lower box dimensions are equal, then we denote the common value by $\dim_\text{B}E$ and call it the box dimension.  For each similarity map $S_i$ in our defining IFS, let $c_i \in (0,1)$ be the contraction constant, i.e. the value such that
\[
| S_i(x) - S_i(y) | = c_i | x-y |
\]
for all $x,y \in X$. The \emph{similarity dimension} is the unique solution $s\geq0$ of the Hutchinson-Moran formula
\[
\sum_{i \in \mathcal{I}} c_i^s = 1.
\]
It is well-known that the Hausdorff and box dimensions of the (homogeneous) self-similar set $F_\emptyset \subset \mathbb{R}^n$ are equal and bounded above by $\min\{s, n\}$ with equality occurring  if the open set condition is satisfied, \cite[Chapter 9]{falconer}.  The main result of \cite{fraser1} was that for an inhomogeneous self-similar set we always have
\begin{equation} \label{bounds}
\max \{ \overline{\dim}_\text{B}  F_\emptyset, \ \overline{\dim}_\text{B} C\} \ \leq \ \overline{\dim}_\text{B} F_C \ \leq \ \max \{ s, \ \overline{\dim}_\text{B} C\},
\end{equation}
which proves (\ref{expected}) in many situations, most notably when the open set condition is satisfied.

Let $\mathcal{I}^* = \bigcup_{k\geq1} \mathcal{I}^k$ denote the set of all finite sequences with entries in $\mathcal{I}$ and for $\I= \big(i_1, i_2, \dots, i_k \big) \in \mathcal{I}^*$ write
\[
S_{\I} = S_{i_1} \circ S_{i_2} \circ \dots \circ S_{i_k}
\]
and
\[
c_{\I} = c_{i_1} c_{i_2}  \dots  c_{i_k}
\]
which is the contraction ratio of $S_\I$. The \emph{orbital set} is defined by
\[
\mathcal{O} \ = \ C \ \cup \  \bigcup_{\I \in \mathcal{I}^*} S_{\I} (C)
\]
and it is easy to see that $F_C \ = \  F_\emptyset \cup \mathcal{O} \ = \  \overline{\mathcal{O}}$ (cf. \cite[Lemma 3.9]{ninaphd}).

\section{Inhomogeneous Bernoulli convolutions and failure of (\ref{expected})} \label{resultssection}

We begin by computing the box dimensions of a family of overlapping inhomogeneous self-similar sets based on Bernoulli convolutions.  Fix $\lambda \in (0,1)$, let $X = [0,1]^2$ and let $S_0, S_1 : X \to X$ be defined by
\[
S_0(x) = \lambda x \ \ \text{ and } \ \ S_1(x) = \lambda x + (1-\lambda,0).
\]
To the homogeneous IFS $\{S_0, S_1\}$ associate the condensation set
\[
C = \{0\}\times [0,1]
\]
and observe that $F_\emptyset = [0,1] \times \{0\}$ and so $\dim_\text{B} F_\emptyset = \dim_\text{B} C = 1$.  We will denote the inhomogeneous attractor of this system by $F_C^\lambda$ to emphasise the dependence on $\lambda$. In this section we construct several counterexamples to (\ref{expected}). Our first counterexample makes use of a well known class of algebraic integers known as \emph{Garsia numbers}. We define a Garsia number to be a positive real algebraic integer with norm $\pm 2$, whose conjugates are all of modulus strictly greater than $1.$ Examples of Garsia numbers include $\sqrt[n]{2}$ and $1.76929\ldots$, the appropriate root of $x^{3}-2x-2=0.$ In \cite{Gar} Garsia showed that whenever $\lambda$ is the reciprocal of a Garsia number, then the associated Bernoulli convolution is absolutely continuous with bounded density.

\begin{figure}[H] \label{examples}
	\centering
	\includegraphics[width=160mm]{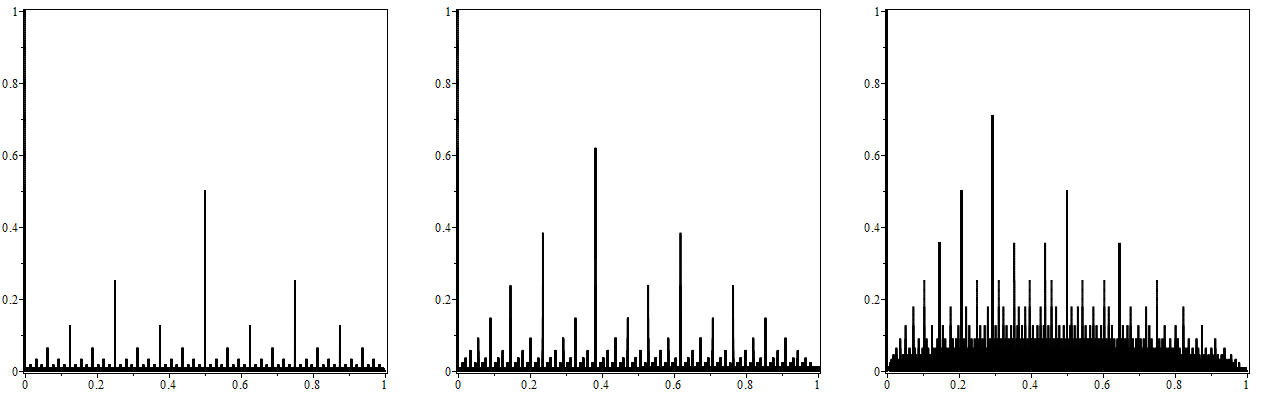}
\caption{Three plots of $F_C^\lambda$, where $\lambda$ is chosen to be 1/2 (where there are no overlaps), the reciprocal of the golden mean (which is Pisot), and the reciprocal of $\sqrt{2}$ (which is Garsia). }
\end{figure}

\begin{thm} \label{thmbernoulli}
If $\lambda\in(1/2,1)$ is the reciprocal of a Garsia number, then
\[
\dim_\text{\emph{B}} F_C^\lambda \ = \  \frac{\log( 4\lambda)}{\log2}  \ > \ 1.
\]
\end{thm}

We defer the proof of Theorem \ref{thmbernoulli} to Section \ref{thmbernoulliproof} below.  For every $\lambda\in(1/2,1)$ which is the reciprocal of a Garsia number, the set $F_C^\lambda$ provides a counterexample to (\ref{expected}) for the upper (and lower) box dimension, but it is also worth noting that this example is `sharp' in that, given the data: $\overline{\dim}_\text{B} F_\emptyset = 1$, $\overline{\dim}_\text{B} C = 1$ and $s = -\log 2/\log \lambda$, we prove that this is as large as $\overline{\dim}_\text{B} F_C$ can be.  For more details, see Corollary \ref{cor2} and Remark \ref{remarksharp}.

Our second source of counterexamples to (\ref{expected}) is a much larger set. As the following statement shows, $F_C^\lambda$ typically provides a counterexample to (\ref{expected}) whenever $\lambda$ lies in a certain subinterval of $(1/2,1)$.

\begin{thm}\label{thmgenericbernoulli}
For Lebesgue almost every $\lambda\in(1/2,0.668)$ we have
\[
\dim_\text{\emph{B}} F_C^\lambda \ = \  \frac{\log( 4\lambda)}{\log2}  .
\]
\end{thm}
The appearance of the quantity $0.668$ is a consequence of transversality arguments used in \cite{BenSol}. Our proof of Theorem \ref{thmgenericbernoulli} will rely on counting estimates appearing in this paper and will be given in Section \ref{thmgenericbernoulliproof}.  We note that the value $\log(4\lambda)/\log2$ also appears as the dimension of a related family of sets.  In particular, for $\lambda\in(1/2,1)$, let $A^\lambda$ be the (homogeneous) self-\emph{affine} set associated to the IFS consisting of affine maps $T_0,T_1: X \to X$ defined by
\[
T_0(x,y) = (\lambda x,y/2) \ \ \text{ and } \ \ T_1(x) = (\lambda x + 1-\lambda, y/2+1/2).
\]
It follows from standard dimension formulae for self-affine sets that the box dimension of $A^\lambda$ is given by $\log(4\lambda)/\log2$ for \emph{every} $\lambda\in(1/2,1)$, see for example \cite[Corollary 2.7]{fraser0}.  Also, for \emph{every} $\lambda \in (0,1/2]$, we have $\dim_\text{B} F_C^\lambda = \dim_\text{B} A^\lambda = 1$, but this case is not so interesting because the IFS defining $F_C^\lambda$ does not have overlaps.  The relevance of this comparison is purely aesthetic, noting that the projection of this IFS onto the first coordinate gives the Bernoulli convolution and onto the second coordinate gives a simple IFS of similarities yielding $C$ as the attractor.

\begin{figure}[H] \label{affine}
	\centering
	\includegraphics[width=160mm]{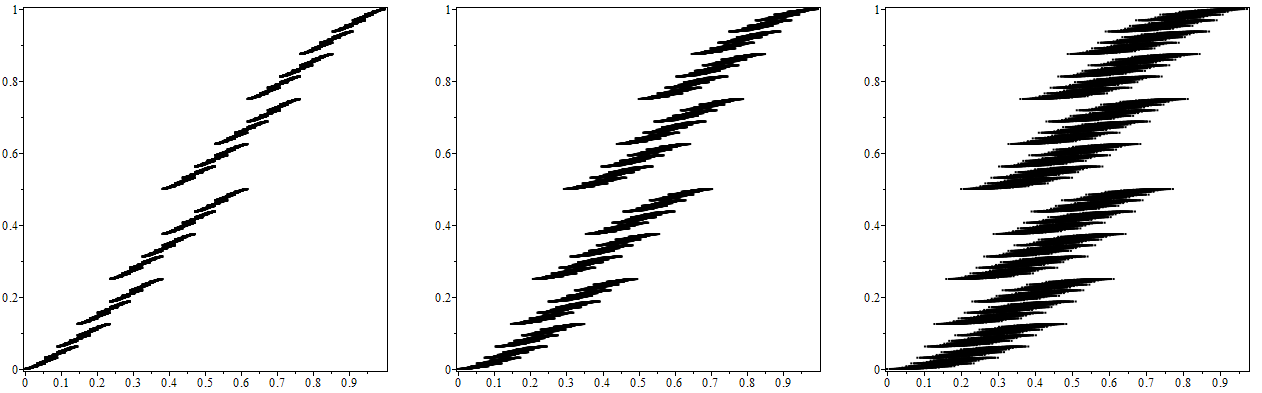}
\caption{Three plots of $A^\lambda$, where $\lambda$ is chosen to be the reciprocal of the golden mean, the reciprocal of $\sqrt{2}$, and $4/5$.}
\end{figure}

The key reason that the sets $F_C^\lambda$ provide counterexamples to  (\ref{expected}) is that the set $F_\emptyset$ is trapped in a proper subspace.  The underlying IFS has potential to give rise to an attractor with dimension bigger than 1, but cannot because the attractor must lie in a 1 dimensional line.  However, since the condensation set does not lie in this subspace, it ``releases'' some of this potential dimension.  At first sight, one might expect this to be the only way to violate (\ref{expected}), so we briefly point out that this is not the case.  In particular, to the system considered in this section, add the similarity map $S_2: x \mapsto \varepsilon x +(0,1-\varepsilon)$ where $\varepsilon$ is chosen in $(0,1-\lambda)$.  Let $E_C^\lambda$ and $E_\emptyset^\lambda$  denote the corresponding inhomogeneous and homogeneous attractors of this system, respectively.  Note that $F_C^\lambda \subset E_C^\lambda$ and so
\[
\overline{\dim}_\text{B} E_C^\lambda \ \geq \  \frac{\log( 4\lambda)}{\log2},
\]
whenever $\lambda$ is the reciprocal of a Garsia number by Theorem \ref{thmbernoulli}.  Moreover, $E_\emptyset^\lambda \subset \pi_1(E_\emptyset^\lambda) \times \pi_2(E_\emptyset^\lambda)$, where $\pi_1$ and $\pi_2$ denote projection onto the first and second coordinates respectively. Both projections are easy to understand since $\pi_1(E_\emptyset^\lambda) = [0,1]$ and $\pi_2(E_\emptyset^\lambda) \subset [0,1]$ is the homogeneous self-similar attractor of $\{ x \mapsto \lambda x, \, x \mapsto \varepsilon x+ (1-\varepsilon)\}$ and so has upper box dimension bounded above by the solution $s = s(\varepsilon,\lambda)>0$ of $\lambda^s+\varepsilon^s=1$.  Using standard properties of the box dimension of products \cite[Chapter 7]{falconer}, this guarantees $\overline{\dim}_\text{B}  E^\lambda_\emptyset \leq 1 +s(\varepsilon,\lambda)$.  By choosing $\varepsilon$ sufficiently small (after fixing $\lambda$), $ s(\varepsilon,\lambda)$ can be made arbitrarily small, in particular to guarantee
\[
\max \{ \overline{\dim}_\text{B}  E^\lambda_\emptyset, \ \overline{\dim}_\text{B} C\} \ \leq \  1+s(\varepsilon,\lambda) \ < \ \frac{\log( 4\lambda)}{\log2} \ \leq \ \overline{\dim}_\text{B} E_C^\lambda
\]
and so (\ref{expected}) fails despite the fact that $E_C^\lambda$ is not contained in a subspace.  For more discussion on possible mechanisms for violating (\ref{expected}), see Section \ref{WSPsection}.

\subsection{Notational remark}

For real-valued functions $A$ and $B$, we will write $A(x) \lesssim B(x)$ if there exists a constant $c>0$ independent of the variable $x$ such that $A(x) \leq c B(x)$, $A(x) \gtrsim B(x)$ if there exists a constant $c'>0$  independent of the variable $x$ such that $A(x) \geq c' B(x)$ and $A(x) \asymp B(x)$ if $A(x) \lesssim B(x)$ and $A(x) \gtrsim B(x)$.  In our setting, $x$ is normally some $\delta>0$ from the definition of box dimension or some $k \in \mathbb{N}$ and the comparison constant $c,c'$ can depend on fixed quantities only, like $\lambda$ and the defining parameters in the IFS.

\subsection{Proof of Theorem \ref{thmbernoulli}} \label{thmbernoulliproof}

Before we get to the proof we state a useful separation property that holds for the reciprocals of Garsia numbers and demonstrate the relevance to our situation via (\ref{S(C) equation}) below.

\begin{lma}[Garsia \cite{Gar}]
\label{Garsia's lemma}
Let $\lambda\in(1/2,1)$ be the reciprocal of a Garsia number and $(i_{k})_{k=1}^{n},(i_{k}')_{k=1}^{n}\in\{0,1\}^{n}$ be distinct words of length $n$. Then $$\Big|(1-\lambda)\sum_{k=1}^{n}i_{k}\lambda^{k-1}-(1-\lambda)\sum_{k=1}^{n}i_{k}'\lambda^{k-1}\Big|> \frac{K}{2^{n}}.$$ For some strictly positive constant $K$ that only depends on $\lambda.$
\end{lma}

This lemma is due to Garsia \cite{Gar}.  For a short self-contained proof of this fact we refer the reader to \cite[Lemma 3.1]{Bak1}.  We note that for any $\I=(i_{1},\ldots,i_{n})\in\{0,1\}^{n}$ we have
\begin{equation}\label{S(C) equation}
S_{\I}(C)=\{S_{\I}(0,0)\} \times  [0,\lambda^{n}] =  \left\{ (1-\lambda)\sum_{k=1}^{n}i_k\lambda^{k-1}\right\} \times [0,\lambda^{n}].
\end{equation} Combining Lemma \ref{Garsia's lemma} with (\ref{S(C) equation}), we see that whenever $\lambda$ is the reciprocal of a Garsia number the images of $C$ will be separated by a factor $K\cdot2^{-n}$. This property is the main tool we use in our proof of Theorem \ref{thmbernoulli}.

\begin{proof}[Proof of Theorem \ref{thmbernoulli}]
 Fix $\delta>0$ and decompose the unit square into horizontal strips of height $\lambda^k-\lambda^{k+1}$ for $k$ ranging from 0 to $k(\lambda, \delta)$, defined to be the largest integer satisfying $\lambda^{k(\lambda, \delta)+1} > \delta$.  Observe that the only part of $F_C^\lambda$ which intersect the interior of the $k$th vertical strip is
\[
\bigcup_{l=0}^{k} \ \bigcup_{\I \in \{1,2\}^l} S_{\I}(C)
\]
which is a union of vertical lines.  Within the $k$th vertical strip, each line intersects on the order of $\lambda^k/\delta$ squares from the $\delta$-mesh.  However, if two of these lines are too close to one another, they may not both contribute to the total intersections for $F_C^\lambda$.  In fact, the number of lines which make a contribution to the total intersections is on the order of $N_\delta( \Lambda(k) )$, where
\[
\Lambda(k) \ = \ \bigcup_{l=0}^{k} \ \bigcup_{\I \in \{1,2\}^l} S_{\I}(0,0),
\]
i.e. the number of base points of the lines intersecting the $k$th vertical strip which lie in different $\delta$-intervals.  This yields
\begin{eqnarray*}
N_\delta\big( F_C^\lambda \big) &\asymp& \delta^{-1} \ + \ \sum_{k=0}^{k(\lambda, \delta)} \big(\lambda^k/\delta \big) \, N_\delta( \Lambda(k) )
\end{eqnarray*}
where the $\delta^{-1}$ comes from the intersections below the $k(\lambda, \delta)$th strip.  It follows from Lemma \ref{Garsia's lemma} and subsequent discussion that
\[
N_\delta( \Lambda(k) )  \ \asymp \  \min\{2^k,\delta^{-1}\}
\]
which is the maximum value possible and where the `comparison constants' are independent of $\delta$ and $k$, but do depend on $\lambda$, which is fixed.  Let $k_0(\delta)$ be the largest integer satisfying $2^{k_0(\delta)}<\delta^{-1}$.  It follows that
\begin{eqnarray*}
N_\delta\big( F_C^\lambda \big) &\asymp& \delta^{-1} \ + \ \sum_{k=0}^{k_0(\delta)} \big(\lambda^k/\delta \big) \, 2^k   \ + \ \sum_{k=k_0(\delta)+1}^{k(\lambda, \delta)} \big(\lambda^k/\delta \big) \, \delta^{-1} \\ \\
&=& \delta^{-1} \ + \ \delta^{-1} \, \sum_{k=0}^{k_0(\delta)}(2\lambda)^k   \ + \ \delta^{-2} \, \sum_{k=k_0(\delta)+1}^{k(\lambda, \delta)} \lambda^k \\ \\
&\asymp&\delta^{-1} \ + \  \delta^{-1} \, (2\lambda)^{k_0( \delta)} \ + \ \delta^{-2}\big(\lambda^{k(\lambda,\delta)} \, - \, \lambda^{k_0(\delta)} \big) \\ \\
&\asymp&\delta^{-1} \ + \  \delta^{-1} \,\delta^{-\log(2\lambda)/\log2} \ + \ \delta^{-2}\big(\delta \, - \, \delta^{-\log\lambda/\log2} \big) \\ \\
&\asymp&  \delta^{-1 -\log(2\lambda)/\log2}
\end{eqnarray*}
which yields
\[
\overline{\dim}_\text{B} F_C^\lambda \ = \ \underline{\dim}_\text{B} F_C^\lambda \ = \ 1 + \log(2\lambda)/\log2 \ = \  \frac{\log( 4\lambda)}{\log2}
\]
as required.
\end{proof}

The above proof actually yields the following stronger result than Theorem \ref{thmbernoulli} which gives the exact rate of convergence to the box dimension.  This type of problem was considered by Lalley \cite{lalley} in the context of homogeneous self-similar sets satisfying the open set condition, where the same rate was obtained.

\begin{thm} \label{exactrate}
If $\lambda\in(1/2,1)$ is the reciprocal of a Garsia number, then
\[
\frac{\log N_\delta\big(F_C^\lambda\big)}{-\log \delta}  \ =  \ \frac{\log( 4\lambda)}{\log2} \, + \, O\Big(\frac{1}{-\log\delta} \Big).
\]
\end{thm}

\subsection{Proof of Theorem \ref{thmgenericbernoulli}}\label{thmgenericbernoulliproof}

To prove this theorem we show that a separation condition similar to that stated in Lemma \ref{Garsia's lemma} holds for a generic $\lambda\in(1/2,0.668).$ Throughout this section we let $\mathcal{J}:=(1/2,0.668)$. Before we state our separation condition we recall some results from \cite{Bak} and \cite{BenSol}.

In \cite{BenSol} Benjamini and Solomyak study the distribution of the set $$A_{n}(\lambda):=\Big\{(1-\lambda)\sum_{k=1}^{n}i_{k}\lambda^{k-1}:(i_{k})\in\{0,1\}^{n}\Big\}.$$ Given $s>0$, $\lambda\in(1/2,1)$, and $n\in\mathbb{N}$ they associate the set
$$R_{2}(s,\lambda,n):=\Big\{(a,b)\in A_{n}(\lambda)^{2}:\, a\neq b ,\, |a-b| \leq \frac{s}{2^{n}}\Big\}.$$ They conjectured that for almost every $\lambda\in(1/2,1),$ there exists $c,C>0$ such that $$cs\leq \frac{\# R_{2}(s,\lambda,n)}{2^{n}}\leq Cs$$ for all $n\in\mathbb{N}$ and $s>0.$ In \cite{BenSol} the authors did not prove this conjecture, however, they did prove several results which that it is true, one of these which is the following.

\begin{thm}[Theorem 2.1 from \cite{BenSol}]
\label{BenSolthm}
There exists $C_{1}>0$ such that $$\int_{\mathcal{J}}\frac{\# R_{2}(s,\lambda,n)}{2^{n}}\leq C_{1}s$$ for all $n\in\mathbb{N}$ and $s>0$.
\end{thm}Theorem \ref{BenSolthm} will be essential when it comes to showing that a generic $\lambda\in\mathcal{J}$ satisfies a separation property. Importantly the $C_{1}$ appearing in Theorem \ref{BenSolthm} does not depend on $n$ or $s$. In \cite{Bak} the first author studied the approximation properties of $\beta$-expansions. To understand these properties the following set was studied

$$T(s,\lambda,n):=\Big\{a\in A_{n}(\lambda): \exists b\in A_{n}(\lambda) \textrm{ satisfying } a\neq b\textrm{ and } |a-b|\leq \frac{s}{2^{n}}\Big\}.$$ In \cite{Bak} it was shown that
\begin{equation}
\label{counting bounds}
\#T(s,\lambda,n)\leq \#R_{2}(s,\lambda,n).
\end{equation} If $T(s,\lambda,n)$ is a small set then the elements of $A_{n}(\lambda)$ are well spread out within $[0,1].$ As was seen in the proof of Theorem \ref{thmbernoulli}, if the elements of $A_{n}(\lambda)$ are well spread out then $F_C^\lambda$ can be a counterexample to (\ref{expected}). We do not show that a separation condition as strong as Lemma \ref{Garsia's lemma} holds for a generic $\lambda\in\mathcal{J},$ but we can prove a weaker condition holds, a condition which turns out to be sufficient to prove Theorem \ref{thmgenericbernoulli}.

\begin{prop}
\label{Countingprop}
For Lebesgue almost every $\lambda\in\mathcal{J},$ the following inequality holds for all but finitely many $n\in\mathbb{N}$:
$$2^{n-1}\leq \#\Big\{a\in A_{n}(\lambda):|a-b|> \frac{1}{n^{2}2^{n}} \textrm{ for all }b\in A_{n}(\lambda)\setminus\{a\}\Big\}.$$
\end{prop} To prove Proposition \ref{Countingprop} we use the Borel-Cantelli lemma and the counting bounds provided by Theorem \ref{BenSolthm}. The following lemma gives an upper bound on the Lebesgue measure of the set of $\lambda$ which exhibit contrary behaviour to that described in Proposition \ref{Countingprop}. The proof of this lemma is based upon an argument given in \cite{Bak}. We write $\mathcal{L}$ for Lebesgue measure.

\begin{lma}
\label{measurebounds}
We have
$$\mathcal{L}\Big(\lambda\in\mathcal{J}:2^{n-1}\leq \#T(n^{-2},\lambda,n)\Big)\leq \frac{2C_{1}}{n^{2}}.$$
\end{lma}
\begin{proof}
Observe that
\begin{eqnarray*}
\mathcal{L}\Big(\lambda\in\mathcal{J}:2^{n-1}\leq \#T(n^{-2},\lambda,n)\Big)&\leq& \mathcal{L}\Big(\lambda\in\mathcal{J}:2^{n-1}\leq \#R_{2}(n^{-2},\lambda,n)\Big) \qquad \qquad (\textrm{by }(\ref{counting bounds}))\\ \\
&=&\frac{2^{n}\mathcal{L}\Big(\lambda\in\mathcal{J}:2^{n-1}\leq\#R_{2}(n^{-2},\lambda,n)\Big)}{2^{n}}\\ \\
&\leq &2 \int_{\lambda\in\mathcal{J}:2^{n-1}\leq\#R_{2}(n^{-2},\lambda,n)}\frac{\#R_{2}(n^{-2},\lambda,n)}{2^{n}}\\ \\
&\leq& 2 \int_{\mathcal{J}}\frac{\#R_{2}(n^{-2},\lambda,n)}{2^{n}}\\ \\
&\leq& \frac{2C_{1}}{n^{2}} \qquad \qquad  (\textrm{by Theorem } \ref{BenSolthm})
\end{eqnarray*}
as required.
\end{proof}

Applying Lemma \ref{measurebounds} we see that

$$\sum_{n=1}^{\infty}\mathcal{L}\Big(\lambda\in\mathcal{J}:2^{n-1}  \leq   \#T(n^{-2},\lambda,n)\Big) \ \leq  \ \sum_{n=1}^{\infty}\frac{2C_{1}}{n^{2}} \ < \ \infty.$$ Thus, the Borel-Cantelli lemma implies that for almost every $\lambda\in\mathcal{J}$ there exists finitely many $n$ satisfying $2^{n-1}\leq \#T(n^{-2},\lambda,n).$ If $\lambda$ does not satisfy a height one polynomial then $\#A_{n}(\lambda)=2^{n}.$ Combining this statement with the above consequence of the Borel-Cantelli lemma we may conclude Proposition \ref{Countingprop}.

Note that Proposition \ref{Countingprop} implies that for Lebesgue almost every $\lambda\in\mathcal{J}$, there exists a constant $\kappa>0$ for which
\begin{equation}
\label{uniformseparation}
2^{n-1}\leq \#\Big\{a\in A_{n}(\lambda): |a-b|> \frac{\kappa}{n^{2}2^{n}} \textrm{ for all }b\in A_{n}(\lambda)\setminus\{a\}\Big\}
\end{equation} is true for \emph{every} $n\in\mathbb{N}.$ Equation (\ref{uniformseparation}) guarantees that at least $2^{n-1}$ elements of $A_{n}(\lambda)$ will be separated by a factor $\kappa(n^{2}2^{n})^{-1}$ holds for every $n\in\mathbb{N}$. The proof of Theorem \ref{thmgenericbernoulli} now follows similarly to the proof of Theorem \ref{thmbernoulli}, so we only point out the differences.

\begin{proof}[Proof of Theorem \ref{thmgenericbernoulli}]   First note that the upper estimate
\[
N_\delta\big( F_C^\lambda \big)\  \lesssim  \   \delta^{-1 -\log(2\lambda)/\log2}
 \]
 remains valid for any $\lambda \in (1/2,1)$, which yields
 \[
 \overline{\dim}_\text{B} F_C^\lambda \ \leq  \frac{\log( 4\lambda)}{\log2}
 \]
 as required.  Let $\varepsilon>0$ and consider the lower bound. In light of Proposition \ref{Countingprop} and the discussion above we can guarantee that for almost every choice of $\lambda \in \mathcal{J}$ we have
 \[
 N_\delta( \Lambda(k)) \,  \gtrsim \,  \min\{2^k,\delta^{-1}\}
 \]
 as long as
\begin{equation} \label{range00}
 \frac{1}{k^22^k}\,  \geq \,  \delta.
\end{equation}
 Let $k_1(\delta)$ denote the maximum value of $k$ for which (\ref{range00}) holds and note that
 \[
 k_1(\delta) \,  \lesssim \,  \frac{-\log \delta}{(1+\varepsilon)\log2}.
 \]
For $k \leq k_1(\delta)$ we have $2^k \leq \delta^{-1}$ and so
\[
N_\delta\big( F_C^\lambda \big)  \ \gtrsim \  \sum_{k=0}^{k_1(\delta)} \big(\lambda^k/\delta \big) \, 2^k  \ = \   \delta^{-1} \, \sum_{k=0}^{k_1(\delta)} (2\lambda )^k   \ \gtrsim \  \delta^{-1} \,  (2\lambda )^{k_1(\delta)}   \ \gtrsim \  \delta^{-1-\log (2 \lambda)/((1+\varepsilon)\log2)}
\]
which yields
\[
\underline{\dim}_\text{B} F_C^\lambda \ \geq \ 1 + \frac{\log(2\lambda)}{(1+\varepsilon)\log2}.
\]
Since $\varepsilon>0$ was arbitrary, the desired lower bound follows.
\end{proof}

 Note that the weaker separation estimate used in proving Theorem \ref{thmgenericbernoulli} means that we cannot obtain the exact rate of convergence to the box dimension as we did for reciprocals of Garsia numbers in Theorem \ref{exactrate}.

\section{Large inhomogeneous self-similar sets in higher dimensions}\label{example}

If the maps of an IFS share a common fixed point, then the attractor is trivial (a singleton). However, the inhomogeneous attractor can be still be highly non-trivial.
In this section we will exhibit an inhomogeneous self-similar set  in $\R^d$ ($d\geq 3$) of large box dimension such that $F_\emptyset$ is a singleton, and $C$ is also a singleton.
This is not possible in $\R$ or $\R^2$ (see Corollary~\ref{cor3}). Our construction is based on Drinfeld's result \cite{Drinfeld} that $SO(3)$ contains finitely generated subgroups with the spectral gap property.

\begin{theorem}\label{fromsg}
Let $d\ge 3$. There are finitely many rotations $g_1, \ldots, g_k\in SO(d)$ and $\eps>0$ such that for any $x\in \R^d$ with $\norm{x}=1$,
$$N_{\delta}(G^n(x)) \ge \min\Big\{(1+\eps)^n, \ \eps \delta^{-(d-1)}\Big\}$$
where
$$G^n(x)=\{g_\I(x)\,:\,\I\text{ is a multi-index of length }n\}.$$
\end{theorem}
\begin{proof}
Let $S^{d-1}$ be the unit sphere in $\R^d$ endowed with the probability Lebesgue measure.
 Let $L^2(S^{d-1})$ denote the $L^2$ space of real valued functions $f:S^{d-1}\to \R$.
 By Drinfeld \cite{Drinfeld} (when $d=3$) and by Margulis \cite{Margulis} and Sullivan \cite{Sullivan} (when $d\ge 4$), there exist rotations $g_1, \ldots, g_k$ and $\eps>0$ for which the operator
$$A: L^2(S^{d-1}) \to L^2(S^{d-1}),$$
$$Af(x)=\frac{1}{k} \sum_{i=1}^k f(g_i^{-1}(x))
$$
has a spectral gap, namely
$$\norm{Af}_2\le (1-\eps) \norm{f}_2 \quad \text{ whenever }\int f=0.$$

Let $f$ be the function which is $1$ on the $\delta$ neighbourhood of $x$ in $S^{d-1}$, and zero otherwise. Then $\int f\approx \delta^{d-1}$. We have
\beq\label{sg}
\norm{A^nf-\textstyle\int f}_2= \norm{A^n(f-\textstyle\int f)}_2\le (1-\eps)^n \norm{f-\textstyle \int f}_2 \le O(1) (1-\eps)^n \delta^{(d-1)/2}.
\eeq
Notice that $\int A^n f =\int f$. Let $E\subset S^{d-1}$ denote the support of $A^nf$ and $\lambda$ denote the measure of $E$. Observe that $E$ is contained by the $\delta$-neighbourhood of $G^n(x)$, which implies
\beq\label{lnd}
\lambda= O(1)\delta^{d-1} N_\delta(G^n(x)).
\eeq

Then
$$\int_{S^{d-1}\setminus E} (A^nf-{\textstyle\int} f)^2 =   (1-\lambda) ({\textstyle\int} f)^2 $$
and by Cauchy--Schwarz,
$$\int_{E} (A^nf-{\textstyle\int} f)^2 \int_{E} 1 \ge \left(\int_E A^nf-{\textstyle\int}f\right)^2 = (1-\lambda)^2 ({\textstyle\int} f)^2.$$
Combining these two, we obtain
$$\int_{S^{d-1}} (A^nf-{\textstyle\int} f)^2 \ge   \Big(1-\lambda+(1-\lambda)^2/\lambda\Big) ({\textstyle\int} f)^2  = (1/\lambda-1) ({\textstyle\int} f)^2.$$
Comparing this to \eqref{sg} gives
$$1/\lambda \le 1 + O(1)(1-\eps)^{2n} \delta^{-(d-1)}.$$
Using \eqref{lnd} we obtain
$$1/N_{\delta}(G^n(x)) \le O(1)\delta^{d-1} + O(1)(1-\eps)^{2n}.$$
%
%
This implies the theorem (with a different $\eps$).
\end{proof}

\begin{theorem}\label{sgconseq}
Let $d\ge 3$. For every $\eps>0$ there is an IFS of similarities in $\R^d$ such that the attractor, $F_\emptyset$, consists of $1$ point, but
$$\underline{\dim}_\text{\emph{B}}\, F_C\ge d-1-\eps$$
whenever $C$ is a singleton not equal to $F_\emptyset$.
\end{theorem}
\begin{proof}
Let $g_1,\ldots, g_k$ and $\eps>0$ be as in Theorem~\ref{fromsg}. Fix $c<1$ sufficiently close to $1$. Consider the contractive similarities
$$S_i(x)=c\cdot g_i(x).$$
Then $F_\emptyset=\{0\}$ and let  $C=\{x\}$ for some $x\neq 0$, which we may assume satisfies $\norm{x}=1$. Then
$$F_C=\{0\} \cup \bigcup_{n=0}^\infty c^n G^n(x).$$
Hence, for any $m,n \in \mathbb{N},$
$$N_{c^{m}}(F_C) \ge N_{c^{m}}(c^n G^n(x))=N_{c^{m-n}} (G^n(x))\ge \min\Big\{(1+\eps)^n, \eps c^{(n-m)(d-1)}\Big\}.$$
It can be shown that $(1+\eps)^n\geq \eps c^{(n-m)(d-1)}$ for $n=\alpha(c)m+O(1)$ where
\[
\alpha(c):=\frac{(d-1)\log 1/c}{\log(1+\eps)-(d-1)\log c}.
\]
 Note that $\alpha(c)\to 0$ as $c\to 1$. This choice of $n$ then yields
\[
\underline{\dim}_\text{B}(F_C)\geq (1-\alpha(c))(d-1).
\]
Choosing $c$ sufficiently close to $1$ proves our result.
\end{proof}

\begin{remark}
When $C$ and $F_\emptyset$ are singletons, the (lower and upper) box dimension of $F_C$ is always less than $d-1$ (see Corollary~\ref{cor4}).
\end{remark}

\section{New upper bounds}\label{upperbounds}

\subsection{Upper bound for the box dimension of an inhomogeneous self-similar set}
Recall that $S_i:\R^d\to \R^d$ ($i\in \iI$) are contracting similarity maps with scaling ratio $c_i$ defining the self-similar set $F_\emptyset\subset \R^d$. We assume that $C\subset \R^d$.
To simplify notation in this section, we will write:
\begin{align*}
s &= \text{similarity dimension of }F_\emptyset \text{ (with the given similarities)}; \\
\alpha & = \dim_\text{H} F_\emptyset = \dim_\text{B} F_\emptyset;\\
\beta & = \overline{\dim}_\text{B} C.
\end{align*}
Recall that the box dimension of a (homogeneous) self-similar set always exists and equals the Hausdorff dimension, see \cite[Corollary 3.3]{techniques}.

\begin{defi}
Assuming $S_i(x)=c_i M_i(x)+b_i$ for an orthogonal matrix $M_i$ and $b_i\in \R^d$, let $T_i(x)=c_i M_i(x)$. Let $G_C$ be the inhomogeneous self-similar set defined by the maps $T_i$ ($i\in\iI$) and $C$.
\end{defi}

\begin{defi}
For $k\ge 0$ let $\iI_k$ denote the set of those multi-indices $\I$ for which $2^{-k-1}< c_\I \le 2^{-k}$.
\end{defi}


The similarity dimension $s$ gives the bound
\beq\label{simineq}
|\iI_k| \le  2^{(s+o_k(1))k}
\eeq
where we use the standard $o_k$ notation; i.e. $o_k(f(k))/f(k)$ tends to zero as $k\to\infty$ (this sequence may depend on the IFS).  To see this, observe that
\[
|\iI_k| \le \sum_{I \in \iI_k}  \big(2^{k+1} c_I \big)^s \leq 2^{sk+s} \sum_{I \in \iI_k}  c_I^s \leq  2^{sk+s} O(1) = 2^{(s+o_k(1))k}.
\]
From the definition of upper box dimension, we immediately have
\beq\label{boxdim}
\Nn(C)\le 2^{\beta n + o_n(n)}.
\eeq

\begin{defi}
Let $\gamma \ge 0$ be the unique real number for which
$$\limsup_{k\to\infty} |\{T_\I : \I\in \iI_k\}|^{1/k} = 2^{\gamma}.$$
(In fact, this sequence is essentially sub-multiplicative and therefore the limit exists.)
\end{defi}

Note that $\gamma\le s$ by \eqref{simineq}.

\begin{lemma}
If $d=1$ or $d=2$ (that is, $F_C$ is a subset of $\R$ or $\R^2$) \emph{or} the matrices $M_i$ are commuting, then $\gamma=0$.
\end{lemma}
\begin{proof}
If the matrices are commuting then the maps $T_i$ are commuting, hence every $T_\I$ ($\I\in \iI$) is of the form
$$\prod_{i\in \iI} T_i^{n_i}$$
with $c_i^{n_i} \ge 2^{-k-1}$.
Therefore $|\{T_\I : \I\in \iI_k\}|$ is at most polynomial in $k$, implying $\gamma=0$.

If $d=1$ then the maps are commuting. For $d=2$, notice that it is enough to show that $|\{M_\I : \I\in\iI_k\}|^{1/k}\to 0$. This then follows by noticing that rotations of $\R^2$ commute and that $R M = M^{-1} R$ for any reflection $R$ and rotation $M$.
\end{proof}

For sets $X, Y\subset \R^d$, let $X+Y=\{x+y:x\in X, y\in Y\}$ and when $X\subset \R$ let $XY=\{xy: x\in X, y\in Y\}$.
\begin{lemma}\label{sum}
Let $X, Y\subset \R^d$. Then $N_\delta(X+Y) \le 2^d N_\delta(X) N_\delta(Y).$
\end{lemma}
\begin{proof}
Consider the product set $X \times Y \subseteq \mathbb{R}^{d} \times \mathbb{R}^{d}$ and the map $p:\mathbb{R}^{d} \times \mathbb{R}^{d} \to \mathbb{R}^d$ defined by $p(x,y) = x+y$. Let $\{U_i\}_i$ and $\{V_j\}_j$ be the half open $\delta$-cubes intersecting $X$ and $Y$ respectively.  Then $\{p(U_i \times V_j)\}_{i,j}$ is the set of half open $2\delta$-cubes intersecting $X + Y$.  Each of these cubes intersects $2^d$ half open $\delta$-cubes, which proves the result.
\end{proof}

\begin{defi}
Let $$F^k_C=\bigcup_{\I\in \iI_k} S_\I(C) \qquad \text{and} \qquad G^k_C=\bigcup_{\I\in \iI_k} T_\I(C).$$
\end{defi}

\begin{prop}\label{prop1}
Let $0\le k\le n$. Then
\begin{align}
\Nn(F^k_C) & \le 2^{o_n(n)} 2^{ks+(n-k)\beta}\label{1} \\
\Nn(F^k_C) & \le 2^{o_n(n)} 2^{k\alpha + (n-k)d}\label{2} \\
\Nn(F^k_C) & \le 2^{o_n(n)} 2^{n\alpha + (n-k)(\beta+d-1)}\label{3} \\
\Nn(F^k_C) & \le 2^{o_n(n)} 2^{n\alpha+k\gamma+(n-k)\beta}
\label{4}.
\end{align}
\end{prop}
\begin{proof}[Proof of \eqref{1}]

For $\I\in\iI_k$ we have
$$\Nn(S_\I(C)) = \Nnk(2^k\cdot S_\I(C)) \le \Nnk(C).$$
Using \eqref{simineq},
\begin{align*}
\Nn(F^k_C) & \le |\iI_k| \Nnk(C) \le 2^{sk+o_k(k)} \Nnk(C) \\
& \le 2^{sk+o_k(k)} 2^{(n-k)\beta+o_{n-k}(n-k)} \le 2^{o_n(n)} 2^{ks + (n-k)\beta}. \qedhere
\end{align*}
\end{proof}

\begin{proof}[Proof of \eqref{2}.]
Let $B(X, r)$ stand for the $r$-neighbourhood of a set $X$. Assuming $C\subset B(F_\emptyset,r)$, we clearly have $F^k_C\subset B(F_\emptyset, r 2^{-k})$. As $F_\emptyset$ intersects at most $2^{k\alpha+o_k(k)}$ grid cubes of side-length $2^{-k}$, the $r 2^{-k}$-neighbourhood of $F_\emptyset$ intersects  at most $2^{k\alpha+o_k(k)}$ grid cubes of side-length $2^{-k}$ also.  Splitting each of these into $2^{(n-k)d}$ cubes of side-length $2^{-n}$ in the obvious way and replacing $2^{o_k(k)}$ with $2^{o_n(n)}$ yields the result.
\end{proof}

\begin{proof}[Proof of \eqref{3}.]
For simplicity assume that $0\in F_\emptyset$. Then $S_\I(0)\in F_\emptyset$ for all $\I$.

For $\I\in \iI_k$,
$$S_\I(C) = S_\I(0)+T_\I(C) \subset F_\emptyset+ T_\I(C) \subset F_\emptyset + c_\I |C| S^{d-1}$$
where $|C|=\{\norm{x}:x\in C\}\subset \R$ and $S^{d-1}=\{x\in\R^d: \norm{x}=1\}$. Therefore
$$F^k_C \subset F_\emptyset+ \{c_\I: \I\in\iI_k\} |C| S^{d-1},$$
and by Lemma~\ref{sum},
\begin{align}
\Nn(F^k_C) & \le 2^d \Nn(F_\emptyset) \cdot \Nn(\{c_\I: \I\in\iI_k\} |C| S^{d-1}) \nonumber \\
& \le 2^d \Nn(F_\emptyset) \cdot \Nnk(\{ 2^k c_\I: \I\in\iI_k\} |C| S^{d-1}) \nonumber \\
& \le 2^d \cdot \Nn(F_\emptyset) \cdot |\{2^k c_\I: \I\in\iI_k\}|  \cdot \Nnk(|C| S^{d-1}) \label{31}
\end{align}
where in the last step we used that $2^k c_\I\le 1$ for $\I \in \iI_k$.
Since multiplication of scalars is commutative,
$|\{2^k c_\I : \I\in\iI_k\}|$
is polynomial in $k$, bounded by $k^D$ for some $D \in \mathbb{N}$ say. The upper box dimension of $|C|S^{d-1}$ is at most $\beta+d-1$. Therefore \eqref{31} implies
\begin{align*}
\Nn(F^k_C) & \le 2^d  2^{n\alpha+o_n(n)} k^D 2^{(n-k)(\beta+d-1)+o_{n-k}(n-k)} \\
& \le  2^{o_n(n)} 2^{n\alpha + (n-k)(\beta+d-1)} \qedhere.
\end{align*}
\end{proof}

\begin{proof}[Proof of \eqref{4}.]
For $\I\in \iI_k$,
$$S_\I(C) = S_\I(0)+T_\I(C) \subset F_\emptyset+ T_\I(C),$$
so
$$F^k_C\subset F_\emptyset + \bigcup_{\I\in\iI_k} T_\I(C).$$
We again have $\Nn(F_{\emptyset})\le 2^{n\alpha + o_n(n)}$, and
$$\Nn(\cup_{\I\in\iI_k} T_\I(C)) \le |\{T_\I : \I\in \iI_k\}| \Nn(T_\I(C)) \le 2^{k\gamma + o_k(k)} 2^{(n-k)\beta+o_{n-k}(n-k)}.$$
So by Lemma~\ref{sum},
\[
\Nn(F^k_C) \le 2^d 2^{n\alpha + o_n(n)} 2^{k\gamma + o_k(k)} 2^{(n-k)\beta+o_{n-k}(n-k)} \le 2^{o_n(n)} 2^{n\alpha+k\gamma+(n-k)\beta}.\qedhere
\]
\end{proof}

Now we state and prove the main result of this Section.

\begin{theorem}\label{thm1}
Let $F_C\subset \R^d$ be an inhomogeneous self-similar set and let $\alpha, \beta, \gamma$ be as above. Then
\begin{align*}
\overline{\dim}_\text{\emph{B}} F_C \le \max_{0\le x\le 1} \ \min \Big\{
& xs+(1-x)\beta,\\
& x\alpha+(1-x)d,\\
& \alpha+(1-x)(\beta+d-1),\\
& \alpha+x\gamma+(1-x)\beta
\Big\}.
\end{align*}

\end{theorem}
\begin{proof}
For some constant $r$, we have $F_C\subset B(F_\emptyset, r 2^{-n}) \cup \left(\cup_{k=0}^{n-1} F^k_C\right)$. Therefore
 $$\Nn(F_C) \le O(1) \Nn(F_\emptyset) + \sum_{k=0}^{n-1} \Nn(F^k_C) \le O(1) \Nn(F_{\emptyset})+n\max_{k\in\{0,\ldots, n-1\}}\Nn(F^k_C).$$
Since $\Nn(F_\emptyset)\le O(1) \Nn(F^{n-1}_C)$, we obtain
$$\frac{\log \Nn(F_C)}{\log 2^n}\le {\max_{k\in\{0,\ldots, n-1\}}}\frac{\log \Nn(F^k_C)}{\log 2^n}+o_n(1).$$
Using Proposition~\ref{prop1}, the obvious substitution $x=k/n,$ and taking the limit, we conclude the proof.
\end{proof}
We will see that Theorem~\ref{thm1} is sharp in many cases.

\subsection{Corollaries of Theorem~\ref{thm1}}
An immediate corollary is the following.

\begin{cor}\label{cor1} \
\begin{itemize}
\item If $\max(\alpha,\beta)<s$ then $\overline{\dim}_\text{\emph{B}} F_C < s$.
\item If $\max(\alpha,\beta)<d$ then $\overline{\dim}_\text{\emph{B}} F_C< d$.
\end{itemize}
\end{cor}
\begin{proof}
These are clear by using the first two terms inside the minimum in Theorem~\ref{thm1}.
\end{proof}


\begin{cor}\label{cor2}
Assume that $\gamma=0$ (which is the case whenever the orthogonal matrices are commuting, or $d=1$ or $d=2$). Then
$$\overline{\dim}_\text{\emph{B}} F_C \le \max\{\beta, \ \alpha+\beta-\alpha\beta/s\}$$
(where $\alpha=\dim_\text{\emph{H}} F_\emptyset$ and $\beta=\overline{\dim}_\text{\emph{B}}\,C$).
\end{cor}


\begin{proof}[Proof of Corollary~\ref{cor2}.]
We will use the first and the fourth estimate of Theorem~\ref{thm1}.
Let $f(x)=xs+(1-x)\beta$ and $g(x)=\alpha+(1-x)\beta$. Both function are monotone.
We have $f(x)=g(x)$ if $x=\alpha/s$ and at this point, the common value is $\alpha+(1-\alpha/s)\beta=\alpha+\beta-\alpha\beta/s$. We also have $f(0)=\beta$ and $g(1)=\alpha$. The maximum of these three values gives the claimed upper bound.
\end{proof}

\begin{remark} \label{remarksharp}
Corollary~\ref{cor2} shows that the box dimension of the inhomogeneous self-similar set of Theorem~\ref{thmbernoulli} is as large as possible, as it is exactly
$\alpha+\beta-\alpha \beta/s= 2-1/s$. In other words, Corollary~\ref{cor2} is sharp in some sense.
\end{remark}

\erase{
\begin{theorem}\label{thm4}
Let $F_C$ be any inhomogeneous self-similar set in $\R^d$. Then
$$\dimfc \le \max\left(\beta, \ \frac{sd-\alpha\beta}{s+d-\alpha-\beta}\right)
= \max\left(\beta, \  s-\frac{(s-\alpha)(s-\beta)}{s+d-\alpha-\beta}\right)
$$
and
$$\dimfc \le \max\left(\beta, \ \frac{\alpha+\beta-\alpha\beta/s+(d-1)}{1+(d-1)/s}\right).$$
(The second bound is better if and only if $\alpha+\beta<1$.)
\end{theorem}
\begin{proof}
Let $d'$ be either $d$ or $\alpha+\beta+d-1$ (the smaller the better the bound).
Let $$f(x)=xs+(1-x)\beta$$ and $$g(x)=x\alpha+(1-x)d'.$$ Then $\dimfc \le \max_x \min(f(x),g(x))$ by Theorem~\ref{thm1}.
We have $f(x)=g(x)$ if
$x(s-\beta-\alpha+d')=d'-\beta$. At this point,
$$f(x)=g(x)=\frac{(d'-\beta)\alpha+(s-\beta-\alpha+d'-(d'-\beta))d'}{s-\beta-\alpha+d'}=\frac{sd'-\beta\alpha}{s-\beta-\alpha+d'}.$$

We also have $f(0)=\beta$, $g(1)=\alpha$.
Therefore the claim follows.
\end{proof}
}

If $C$ has box dimension zero then we have the following corollary.
\begin{cor}\label{cor3}
Assume that  $\beta=0$ (for example, $C$ consists of a single point). Then
\begin{align*}\overline{\dim}_\text{\emph{B}} F_C & = \alpha = \dim_\text{\emph{H}} F_\emptyset \qquad \text{ if } d=1,2 \text{ or }\gamma=0;\\
\overline{\dim}_\text{\emph{B}} F_C & \le \frac{s}{1+(s-\alpha)/d'} = \frac{d'}{1+(d'-\alpha)/s} \qquad \text{ if } d\ge 3,
\end{align*}
where $d'=\min(d, \ \alpha+d-1)$.
\end{cor}
\begin{proof}
If $d=1,2$ or $\gamma=0$ then this is immediate from Corollary~\ref{cor2}.

Using the first three terms in Theorem~\ref{thm1} with $\beta=0$ we obtain
$$\overline{\dim}_\text{B} F_C \le \max_{0\le x\le 1} \min \Big(xs, \ x\alpha+(1-x)d, \ \alpha+(1-x)(d-1)\Big)=\max_{0\le x\le 1} \min \Big(xs, \ x\alpha+(1-x)d'\Big).$$
Calculating this maximum gives the required result.
\end{proof}

\begin{cor}\label{cor4}
Let $d \geq 2$ and assume that all similarity maps share a common fixed point, that is, $F_\emptyset$ is a singleton and $\alpha=0$. If $C$ is a singleton (or any compact set of box dimension $0$), we have
$$\overline{\dim}_\text{\emph{B}} F_C \le \frac{s}{1+\frac{s}{d-1}}=\frac{d-1}{1+\frac{d-1}{s}}<d-1.$$
\end{cor}
\begin{proof}
Immediate from Corollary~\ref{cor3}.
\end{proof}

\begin{remark}
The $d-1$ bound in Corollary~\ref{cor4} is sharp by Theorem~\ref{sgconseq}.
\end{remark}

\section{The weak separation condition case} \label{WSPsection}

Our next result provides a simple sharpening of (\ref{bounds}), which will allow us to extend the class of inhomogeneous self-similar sets for which we know that (\ref{expected}) holds.  In particular, the separation property required in \cite{fraser1} can be significantly weakened to the \emph{weak separation property}.

For $\I= (i_1, i_2, \dots, i_k ) \in \mathcal{I}^*$, let $\I^\dagger = (i_1, i_2, \dots, i_{k-1} )$ and for $r \in (0,1)$, let
\[
\mathcal{I}(r)  = \{ \I \in \mathcal{I}^*  \ :  \ c_{\I} \leq r < c_{\I^\dagger} \}
\]
be the set of finite strings whose corresponding contraction ratio is approximately $r$.  For convenience we assume the map corresponding to the empty word is the identity with ratio 1.  Let $\sim$ be the relation on $\mathcal{I}^*$ defined by $\I \sim \J$ if $S_{\I} = S_{\J}$.  Thus
\[
\mathcal{I}(r) / \sim
\]
is the set of finite strings whose corresponding similarity maps are distinct and have contraction ratio approximately $r$. In a slight abuse of notation we identify equivalence classes with an arbitrarily chosen representative from the class.  Note that $F_\emptyset$ is the attractor of $\{S_\I\}_{\I \in  \mathcal{I}(r) / \sim}$ for all $r$.  Finally let $\alpha(r)$ be the similarity dimension for this reduced IFS, i.e. the unique solution of
\[
\sum_{\I \in  \mathcal{I}(r) / \sim} c_\I^{\alpha(r)} = 1.
\]
It is straightforward to see that $\alpha(r)$ decreases as $r$ decreases and so we define the \emph{modified similarity dimension} as
\[
s^* \ = \ \lim_{r \to 0} \alpha(r) \ = \ \inf_{r \in (0,1)} \alpha(r)
\]
and note that it is an upper bound for the upper box dimension of $F_\emptyset$. The next theorem is a sharpening of \cite[Theorem 2.1]{fraser1}.

\begin{thm} \label{thmWSP}
If $F_C$ is an inhomogeneous self-similar set, then
\[
\max \{ \overline{\dim}_\text{\emph{B}}  F_\emptyset, \ \overline{\dim}_\text{\emph{B}} C\} \ \leq \ \overline{\dim}_\text{\emph{B}} F_C \ \leq \ \max \{ s^*, \ \overline{\dim}_\text{\emph{B}} C\}.
\]
\end{thm}

\begin{proof}
 It suffices to show that $\overline{\dim}_\text{B} F_C \ \leq \ \max \{ \alpha(r), \ \overline{\dim}_\text{B} C\}$ for all $r \in (0,1)$, recalling that the lower bound is trivial. Fix $r \in (0,1)$ and let
\[
\mathcal{J}(r)  = \{ \I \in \mathcal{I}^* \ : \ \text{$\I$ is a subword of $\I'$ for some $\I' \in \mathcal{I}(r)$} \}.
\]
Let
\[
C(r) \ = \ \bigcup_{\I \in \mathcal{J}(r)} S_{\I}(C) \  \cup \  C
\]
and observe that this is a finite union of compact sets and so is itself compact and, moreover, has upper box dimension equal to that of $C$. This latter fact is due to upper box dimension being stable under taking finite unions and bi-Lipschitz images, see \cite[Chapter 3]{falconer}.  Let $F_{C(r)}$ denote the inhomogeneous attractor of the reduced IFS corresponding to $\mathcal{I}(r) / \sim$ along with the compact condensation set $C(r)$.  It follows from (\ref{bounds}) that
\[
\overline{\dim}_\text{B} F_{C(r)} \ \leq \ \max\{ \alpha(r), \, \overline{\dim}_\text{B} C(r)\} \ = \ \max\{ \alpha(r), \, \overline{\dim}_\text{B} C\}.
\]
Finally, observe that $F_{C(r)} = F_C$ because the extra copies of $C$ found in $C(r)$ precisely fill the gaps left by considering the reduced IFS.
\end{proof}

The \emph{weak separation property} is a much weaker condition than the open set condition and has proved very useful in the study of self-similar sets with overlaps.  It is satisfied if the identity map is \emph{not} an accumulation point of the set
\[
 \{ S_{\I}^{-1} \circ S_{\J} \   : \  \I, \J \in \mathcal{I}^* \}
\]
equipped with the uniform norm, see \cite{zerner}.  Zerner \cite[Theorem 2]{zerner} proved that if $F_\emptyset \subseteq \mathbb{R}^n$ is a self-similar set which does not lie in a hyperplane and the defining IFS satisfies the weak separation property, then $\overline{\dim}_\text{B}  F_\emptyset = s^*$.  This yields the following immediate corollary.

\begin{cor} \label{thmWSP2}
If $F_C$ is an inhomogeneous self-similar set such that the restriction of the underlying IFS to the smallest hyperplane containing $F_\emptyset$ satisfies the weak separation property, then
\[
\overline{\dim}_\text{\emph{B}} F_C \ = \  \max\{ \overline{\dim}_\text{\emph{B}} F_\emptyset , \, \overline{\dim}_\text{\emph{B}} C \}.
\]
\end{cor}

We note that Corollary \ref{thmWSP2} implies that the set $F_C^\lambda$ satisfies (\ref{expected}) if $\lambda$ is such that the underlying IFS satisfies the weak separation property.  This happens if, for example, $\lambda$ is the reciprocal of a Pisot number; see the second example from Figure \ref{examples}.

In light of Theorem \ref{thmWSP}, in order to find examples where (\ref{expected}) fails for $F_C \subseteq \mathbb{R}^n$, one is forced to find examples where $\overline{\dim}_\text{B} F_\emptyset < \min\{s^*,n\}$.  This is linked to an important open conjecture in the study of homogeneous self-similar sets, see for example \cite{problems_update, hochman2}, which states that the only mechanism for the Hausdorff dimension of a self-similar set in $\mathbb{R}$ to be strictly less than $\min\{1, s\}$ is for the IFS to have exact overlaps.  Exact overlaps are precisely what causes $s^* < s$ and so it may be true that the Hausdorff dimension of such a self-similar set is always equal to $\min\{1, s^*\}$.

\begin{conj}
If $F_C \subset \mathbb{R}$ is an inhomogeneous self-similar, then (\ref{expected}) is satisfied, i.e. $\overline{\dim}_\text{\emph{B}} F_C \ = \  \max\{ \overline{\dim}_\text{\emph{B}} F_\emptyset , \, \overline{\dim}_\text{\emph{B}} C \}$.
\end{conj}

Finally we point out that one can deduce several corollaries from the work of Hochman \cite{hochman, hochman2} of the form: ``for some continuously parameterised family of IFSs of similarities, (\ref{expected}) holds almost surely for the corresponding inhomogeneous self-similar set $F_C$ for any compact $C$''.  This is because Hochman gives several such results guaranteeing $\dim_\H F_\emptyset = \min\{s,n\}$ almost surely with respect to the specific parameterisation.  In fact, the results often hold outside exceptional sets of dimension strictly less than the dimension of the parameter space.  Rather than state these explicitly we refer the reader to \cite[Theorem 1.12, 1.13]{hochman2}.

\end{document}